\documentclass{amsart}
\usepackage{amssymb,latexsym,amscd, amsthm, epsfig,amsmath,amssymb}
\usepackage{color,enumerate}

 \newtheorem{thm}{Theorem}[section]
 \newtheorem{cor}[thm]{Corollary}
 \newtheorem{lem}[thm]{Lemma}
 \newtheorem{prop}[thm]{Proposition}
 \theoremstyle{definition}
 \newtheorem{defn}[thm]{Definition}
 \newtheorem{rem}[thm]{Remark}
 \theoremstyle{definition}
 \newtheorem{ex}[thm]{Example}
\newtheorem{nota}[thm]{Notation}

 \newcommand{\M}{\mathcal{M}}
 \newcommand{\N}{\mathcal{N}}
 \newcommand{\E}{\mathcal{E}}

 \newcommand{\PP}{\mathbb{P}}

\usepackage{xypic} 
\usepackage{amsthm}

\begin{document}

\title[Grassmann secant varieties]
{Grassmann secants, identifiability, and linear systems of tensors.}

\author[E. Ballico]{Edoardo Ballico}
\address[E. Ballico]{Dept. of Mathematics,  
University of Trento, 38123 Povo (TN), Italy}
\email{ballico@science.unitn.it}

\author[A. Bernardi]{Alessandra Bernardi}
\address[A. Bernardi]{GALAAD, INRIA M\'editerran\'ee,  
BP 93, F-06902 Sophia Antipolis, France.}
\email{alessandra.bernardi@inria.fr}

\author[M.V.Catalisano]{Maria Virginia Catalisano}
\address[M.V.Catalisano]{DIPTEM - Dipartimento di Ingegneria 
della Produzione, Termo-energetica e Modelli Matematici, 
Universit\`{a} di Genova, Piazzale Kennedy, pad. D 16129 Genoa, Italy.}
\email{catalisano@diptem.unige.it}

\author[L. Chiantini ]{Luca Chiantini}
\address[L. Chiantini]{Universit\`a degli Studi di Siena, 
Dipartimento di Scienze Matematiche e Informatiche, 
Pian dei Mantellini, 44, 53100 Siena, Italy.}
\email{luca.chiantini@unisi.it}

\maketitle


\begin{abstract} For any  irreducible non-degenerate  variety $X\subset \mathbb{P}^r$, we 
give a criterion for the $(k,s)$-identifiability of $X$. If $k\leq s-1 <r$, then the $(k,s)$-identifiability holds for $X$ if and only if  the $s$-identifiability holds for the Segre product $Seg(\mathbb{P}^k\times X)$. Moreover, if the $s$-th secant variety of $X$ is not defective and it does not fill the ambient space, then we can produce a family of pairs $(k,s)$ for which the $(k,s)$-identifiability  holds for  $X$.
\end{abstract}

\section*{Introduction}
 
In 1915, in an elegant XIX-century style Italian language  (\cite[p. 97]{Ter15}), A. Terracini pointed out  that the defectiveness of the $s$-th secant varieties of a Segre product $Seg(\mathbb{P}^k\times V_d)$ between a projective space $\mathbb{P}^k$ and a Veronese surface $V_d$, is related to the fact that  the set of all $\mathbb{P}^k$'s lying in the span of $s$ independent points of $V_d$ does not have the dimension that one can expect, from an obvious count of parameters. That pioneering work (also known as Terracini's second Lemma) was followed by a series of papers by Bronowski,
who studied the simultaneous  expressions of forms in terms of given powers (\cite{Br}). Only many decades later, Terracini's analysis has been rephrased in modern terms. In 2001, C. Dionisi and C. Fontanari proved that a Terracini's like result can be formulated by replacing Veronese surface with any irreducible non-degenerate projective variety $X$ (\cite[Proposition 1.3]{DF}). In analogy with the notion of defectiveness, which is set forth for secant varieties, they utilized  the concept of $(k,s)$-{\it Grassmann defect} that holds for the varieties $X$ for which the set of all $\mathbb{P}^k$'s lying in the span of $s$ independent points of $X$ has  dimension smaller than the expected one. The Zariski closure $GS_X(k,s)$ (see Definition \ref{GrassDef} for more details) of such a set is called {\it $(k,s)$-Grassmann secant variety of $X$}
(see also \cite{ChC}). In this setting, Terracini's idea led to the following general result.

\begin{prop}\cite[Proposition 1.3]{DF} \label{CarlaClaudio}
Let $X\subset  \PP^r$ be an irreducible non-degenerate projective variety of dimension $n$. Then $X$ is $(k, s)$-defective with defect $\delta_{k,s}(X)=\delta$ if and only if $Seg(\PP^k \times X)$ is $s$-defective with defect $\delta_{s}(Seg(\PP^k \times X))=\delta$.
\end{prop}

Actually, it turns out that the relation stated in the previous proposition holds because of the existence of a rational map
$$\Phi: \sigma_{s}(Seg(\mathbb{P}^k\times X))\dashrightarrow   GS_X(w,s),$$
where $w= \min\{s-1,k\}   $,
 that can be precisely described, in terms of coordinates.

The map $\Phi$ determines a general relation between the dimension
of $GS_X(k,s)$ and the dimension of the $s$-th secant 
variety $\sigma_s (Seg(\PP^k\times X))$ of the Segre embedding of 
$\mathbb{P}^k\times X$ into $\mathbb{P}^{rk+r+k}$ (see Theorem \ref{mappaPhi}).
Consequently, by studying $\Phi$, one can compute the dimension of 
$\sigma_{s}(Seg(\mathbb{P}^k\times X))$, in the case of $k \geq s-1$
 (see Theorem \ref{dimsegre}).
It is worth mentioning that such a result and its consequences endorse 
Conjecture 5.5 of \cite{AB09b}.

Notice that, in the recent preprint \cite{Buczynski:2009fk}, 
J. Buczynski and J.M. Landsberg 
obtain  results on the dimension of secant varieties of Segre
embeddings
(\cite[Proposition 3.9]{Buczynski:2009fk}). Their method is based on 
the invariant properties of a rational map $\pi$, which is very close to
our $\Phi$.

Because of the natural way in which the map $\Phi$ arises, 
it is easy to guess that it could apply to the study of
other properties of Segre products. 
This is the reason why we decided to assign $\Phi$
the role of key tool in this paper.
We are able to prove, indeed, that
the map also provides a link between the {\it identifiability}
of $X$ and of $Seg(\PP^k\times X)$ (Theorem \ref{main}).

 A crucial question arising from numerous applications, above all Algebraic Statistics and Signal Processing, is whether or not a set of parameters of a given model is identifiable.
From a purely mathematical point of view, this kind of problem can be stated  in complete generality. Let  $X\subset \mathbb{P}^r$ be any irreducible non-degenerate projective variety.
Given $s$ distinct points $P_1, \ldots , P_s\in X$, fix a projective linear subspace $\Pi \subset \langle P_1, \ldots , P_s \rangle$. How many more $\mathbb{P}^{s-1}$ containing $\Pi$ can be found among those that are $s$-secants to $X$? When the answer to this question is: ``No one besides $\langle P_1, \ldots , P_s\rangle$", then $\Pi$ is said to be $X${\it -identifiable}.  Moreover, if the general  $\mathbb{P}^k$ lying in the span of $s$ independent points of $X$
is contained in a unique $\mathbb{P}^{s-1}$ $s$-secant to $X$, then we say that the $(k,s)${\it -identifiability holds for $X$} (or equivalently that {\it $X$ is $(k,s)$-identifiable}).
The identifiability properties are studied, for the case
$k=0$,  because of their many applications
(see, e.g. \cite{CL}, \cite{Kr}, \cite{DeL}, \cite{AMR}, \cite{CaCh}, \cite{MR1931400},
\cite{BC}, \cite{ChCi}, \cite{Mella}, \cite{CO}, \cite{KB}, \cite{ERSS}, \cite{MR2769871}). In that particular case we will write   {\it $s$-identifiability} instead of $(0,s)$-identifiability.

Our main contribution to this problem giving rise to new results follows
from a direct application of the map $\Phi$ above, and
is summarized in the following two theorems 
(see Theorems \ref{main} and  \ref{tre} respectively).
\\
\\
 \textbf{Theorem} {\it Let $k \leq s-1 <r$. The 
 $(k,s)$\emph{-identifiability} holds for $X$ if and only if
the $s$-identifiability holds for  $Seg(\mathbb{P}^k\times X)$.
}\\
\\
 \textbf{Theorem}
 {\it Let $s$ be an integer such that 
$r>sn+s-1$ and $X$ is not $s$-defective.
Then for all integers $k>0$, $k\leq s-1$ such that  
$$sn + (k + 1)(s-1 - k)< (k + 1)(r - k)$$
the $(k,s)$-identifiability holds for $X$.
}
\\
\\
In the particular case in which the variety $X$ itself is a standard Segre variety $Seg(\mathbb{P}(V_1)\times \cdots \times \mathbb{P}(V_t))$,  for certain  vector spaces $V_i$, then the identifiability properties allow one to deduce  many peculiar examples on linear systems of tensors $\mathcal{E}\subset \mathbb{P}(V_1 \otimes \cdots \otimes V_t)$ (see Section \ref{linsyst}). As an example (see Example \ref{4x4}),  we can also show that the general linear system of dimension $3$ of matrices of type $4\times 4$ and rank $s= 6$ is not identifiable and it is computed by exactly two sets of decomposable tensors.

Actually, Theorem  \ref{main}  can be viewed  as the identifiability version of \cite[Proposition 1.3]{DF}.
We are  persuaded that  the rational map $\Phi$  will be a key tool for
even further investigations on secant varieties and applications.  We
refer to the remark at the end of Example \ref {5.7},
which supports
our guess.

After the preliminary Section \ref{prelim} where we introduce all the needed notions in complete generality for any irreducible non-degenerate projective variety $Y$, we devote  all of Section \ref{TheMap} to a detailed description of $\Phi$.  In Section \ref{ident} we make use of $\Phi$ to study the identifiability properties of $Seg(\mathbb{P}^k\times X)$ that will be applied in Section \ref{linsyst} for the particular case of linear systems of tensors. Observe that both Theorem \ref{main} and  Theorem \ref{tre} quoted above that are the main new theorems of this paper, are obtained by a use of  $\Phi$. The same belief in the importance of $\Phi$ urges us to
recover  \cite[Prop. 3.9]{Buczynski:2009fk} via a straightforward use of $\Phi$ (Theorem  \ref{dimsegre}).   
At the end of  Section \ref{dove}, we  also show some new, interesting
consequences of them,  in the particular case of $X$ being a
Segre-Veronese variety.

\section{Preliminaries, Notation and Basic Definitions}\label{prelim}

Throughout this paper we will always work over an algebraically closed 
field $K$ of characteristic 0. 
All the definitions that we give in this section holds for any irreducible non-degenerate projective variety
$Y$  contained in $\mathbb{P}^m$. 

Let us recall the classical definition of {\it secant varieties} and the more modern concept of {\it Grassmann secant varieties}.

\begin{defn}\label{secant} 
The $s$-th {\it higher secant variety} $\sigma_{s}(Y)$ of $Y$, is
 the Zariski closure of the union of all projective linear spaces spanned by  
$s$ distinct points of $Y$:
$$\sigma_{s}(Y):= \overline{
 \bigcup_{P_{1}, \ldots , 
P_{s}\in  Y} \langle P_{1}, \ldots , P_{s} \rangle}\subset \PP^m.$$
The expected dimension of $\sigma_{s}(Y)$  is
\begin{equation} \label{expdimsigma1}
exp\dim \sigma_s(Y): = \min\{ s(\dim Y+1)-1; m\}.
\end{equation}
When $\sigma_s(Y)$ does not have the expected dimension, $Y$ is
said to be $s$-{\it defective}, and the positive integer
$$
\delta _{s}(Y): = exp\dim \sigma_s(Y)-\dim \sigma_s(Y)
$$
is called the  $s${\it-defect} of $Y$. 
\end{defn} 

\def\Terracini1{
The fact that $\sigma_s(Y)$ can have dimension smaller than 
the expected one depends on the fact that $s$ generic tangent 
spaces to $Y$ can have non trivial intersection. This is clearly 
explained by the well known Terracini's Lemma (the first one).

\begin{lem}[{\bf Terracini's Lemma}]\label{TerrLemma}
 Let $Y\subset \mathbb P^r$  be a projective, irreducible
variety, and let $P_1,\ldots ,P_s \in Y$  be $s$ generic distinct
points. Then the projectivized tangent space to 
$\sigma_{s}(Y)$ at a generic point $Q \in \langle P_1,\ldots ,
P_s \rangle$  is the linear span in $\PP^r$ of the tangent spaces 
$T_{Y, P_i}$ to $Y$ at the points $P_i$, $i=1,\ldots ,s$, hence
$$ \dim \sigma_{s}(Y) = \dim  \langle T_{Y,P_1},\ldots ,
T_{Y,P_s}\rangle.$$
\end{lem}

Let $P_1,\ldots ,P_s \in Y\subset \mathbb{P}^r$  be generic distinct
points as above and consider the $0$-dimensional scheme of  $s$ 2-{\it fat points} $Z\subset Y$ defined by the ideal sheaf ${\mathcal I}_{Z} = 
{\mathcal I}^2_{P_1}\cap  \cdots \cap {\mathcal I}^2_{P_s}
\subset {\mathcal O}_Y$. The  bijection between the elements of $H^0(Y,{\mathcal I}_{Z}(1))$ and the hyperplanes of $\PP^r$ containing  the linear space $ \langle 
 T_{Y,P_1},\ldots ,T_{Y,P_s}\rangle$, leads to the following important corollary.

\begin{cor} \label{TerrCor}  Let $Y\subset \mathbb{P}^r$ and $Z\subset Y$ be as above. 
Therefore
$$ \dim \sigma_{s}(Y) = r - 
\dim H^0(Y,{\mathcal I}_{Z}(1)) = H(Z,1)-1,$$
where  $H(Z,-)$ is the Hilbert function of $Z$.
\end{cor}

From such a result we obtain the following obvious, 
but useful, remark:

\begin{rem} \label{TerrRem}  Let $Z\subset Y\subset \mathbb{P}^r$, be as above, the scheme defined by  ${\mathcal I}_{Z} = 
{\mathcal I}^2_{P_1}\cap  \cdots \cap {\mathcal I}^2_{P_s}
\subset {\mathcal O}_Y$. 
\begin{enumerate}[i)]
\item If $ H(Z,1) = r+1$, then $ \dim \sigma_{t}(Y) = r$  for all 
 $t > s$.
\item If $\ H(Z,1)  = s(n +1)$, then $ \dim \sigma_{t}(Y) =
t (n +1) -1 $ for all  $t < s$.
\end{enumerate}
\end{rem}
}

The fact that $\sigma_s(Y)$ can have dimension smaller than 
the expected one,  is clearly 
explained by the well known Terracini's Lemma (the first one). We remark here a consequence that arises when interpreting  Terracini's Lemma in terms of fat points (see \cite[Section 2]{CGG}).

\begin{rem} \label{TerrRem} Let $P_1,\ldots ,P_s \in Y\subset \mathbb{P}^m$  be generic distinct
points, consider the $0$-dimensional scheme of  $s$ 2-{\it fat points} $Z\subset Y$ defined by the ideal sheaf ${\mathcal I}_{Z} = 
{\mathcal I}^2_{P_1}\cap  \cdots \cap {\mathcal I}^2_{P_s}
\subset {\mathcal O}_Y$ and denote by $H(Z,d)$ the Hilbert function of $Z$ in degree $d$.

\begin{enumerate}[i)]
\item If $ H(Z,1) = m+1$, then $ \dim \sigma_{t}(Y) = m$  for all 
 $t \geq s$.
\item If $\ H(Z,1)  = s(\dim Y +1)$, then $ \dim \sigma_{t}(Y) =
t (\dim Y  +1) -1 $ for all  $t \leq s$.
\end{enumerate}
\end{rem}

The following definition is due to \cite{ChC}.

\begin{defn}\label{GrassDef}    
Let 
$0 \leq k \leq s-1 \leq m$
 be integers and let $\mathbb{G}(k,m)$ 
be the Grassmannian of linear $k$-spaces contained in $\PP^m$.

The {\it $(k,s)$-Grassmann secant variety of $Y$},  
denoted with $GS_Y(k,s)$, is the closure in $ \mathbb{G}(k,m)$ of the set
$$\{\Lambda \in \mathbb{G}(k, m) |   \Lambda 
\hbox{ lies in the linear span of } 
 s \ \hbox  {independent points of } Y \}.$$
\end{defn}

Notice that, for $k=0$, the Grassmann secant variety $GS_Y(k,s)$
coincides with the secant variety $\sigma_{s}(Y)$.

The expected dimension of $GS_Y(k,s)$ is, as always, the minimum between the dimension of the ambient space and the obvious count of parameters obtained by choosing $s$ points on $Y$, and a point in the Grassmannian $G(k,s-1)$ (see eg. \cite{CC08}), that is,
\begin{equation} \label{expdimGS}
exp\dim GS_Y(k,s)=\min\{s (\dim Y) + (k + 1)(s-1 - k); \ (k + 1)(m - k)\}.
\end{equation}
In analogy with the theory of classical secant varieties,  
we define the  \emph{$(k,s)$-defect} of $Y$ as the integer: 
$$
\delta_{k,s}(Y):= exp\dim GS_Y(k,s)-\dim
GS_Y(k,s). 
$$

We end this section by introducing the  concept of {\it identifiability} which will be the core of Sections \ref{ident} and \ref{linsyst}.

\begin{defn} \label{computed} Fix a linear subspace 
$\Pi\subset \PP^m$ (possibly a point) and let 
$P_1,\dots,P_s\in Y$ be distinct points. We say that $\Pi$ is 
\emph{computed by} $P_1,\dots,P_s\in Y$ if $\Pi$ belongs to the 
linear span of the points $P_i$'s.

In this case, we say that $P_1,\dots, P_s$ provide a 
{\it decomposition} of $\Pi$.

The minimum integer $s$ for which 
there exist $s$ distinct points $P_1, \ldots,$ $ P_s \in Y$ 
such that $\Pi$ is computed by $P_1,\dots,P_s$, is called the 
{\it  $Y$-rank of  $\Pi$}.  We indicate it with $r_Y(\Pi)$.
\end{defn}

\begin{defn}\label{IdentSub} Let $Y$ and  $\Pi$ be as in 
Definition \ref {computed} and let $s$ be the $Y$-rank of $\Pi$. 
We say that $\Pi$ is $Y$-\emph{identifiable} if there is 
a unique set of distinct points $\{P_1, \ldots , P_s\}\subset Y$ 
whose span contains $\Pi$.
\end{defn}

\begin{defn}\label{IProperty}  Let $Y\subset\mathbb{P}^m$ as above.  We say that the $(k,s)$\emph{-identifiability} 
holds for $Y$ if the general element of $GS_Y(k,s)$ 
has $Y$-rank equal to $s$ and it  is $Y$-identifiable. 

When $k=0$, we will often omit $k$ and we will simply 
 say that the {\it $s$-identifiability} holds for $Y$.
\end{defn}

\section{The map $\Phi$}\label{TheMap}

From now on, with $X$ we will always denote an irreducible non-degenerate projective variety of dimension $n$ contained in $\mathbb{P}^r$. For any integer $k\geq 0 $, set $N=rk+r+k$ and let $\varphi :\PP^k \times X \rightarrow \PP^N$ be the  Segre embedding of $\PP^k \times X$. The image of $\varphi$ is the Segre variety $Seg(\PP^k \times X)\subset \mathbb{P}^N$.

The aim of this section is to study a projective rational map $\Phi=\Phi(X,k,s)$  from the $s$-th secant variety $\sigma_s(\PP^k \times X)$ of  the Segre variety  $Seg(\PP^k \times X)$  into  the  Grassmann secant variety $GS_X(w,s)$, where $w= \min\{s-1,k\}$. What we do until Lemma 2.2 can be compared  with what is done in \cite{Buczynski:2009fk} for the particular case of $k\leq r$ and $s\geq k+1$. In fact, in \cite[Corollary 3.6]{Buczynski:2009fk} the authors consider the case of $k\leq r$ and they introduce a rational map $\pi$ from the projective space $\mathbb{P}^N$ to the Grassmannian $G(k,r)$ whose restriction to $\sigma_s(\PP^k \times X)$ when $s\geq k+1$ maps onto $GS_X(w,s)$.

We will give a definition of the map $\Phi$, in terms
of local coordinates. Then, we will show how it allows  to 
link the main secant properties of $Seg(\PP^k\times X)$
with the Grassmann-secant properties of $X$.

\begin{nota}
For any choice of $t$ points 
$$\mathcal {A}_i=(a_{i,0},\dots,a_{i,r})  \in K^{r+1},\quad 1 \leq i \leq t,$$ 
we denote by  $(\mathcal {A}_1, \dots, \mathcal { A}_t)$
the $t(r+1)$-uple  
$$ (a_{1,0},\dots,a_{1,r},\dots,a_{t,0},\dots,a_{t,r}) \in K^{t(r+1)}.$$

Let $(\lambda_{0}, \ldots , \lambda_{k})$ and $(x_0,\dots,x_r)$ 
be  sets of homogeneous coordinates for the  points 
${\Lambda}\in \PP^k$ and  $P \in  X$, respectively.

Consider the point $\varphi({\Lambda}, P)  \in Seg(\PP^k\times X)$, 
so that, in coordinates:
  $$\varphi({\Lambda}, P) 
  = (\lambda_0x_0,\dots, \lambda_0x_r,
 \lambda_1x_0,\dots, \lambda_1x_r,\dots,
 \lambda_kx_0,\dots, \lambda_kx_r).$$

Accordingly with the previous notation, we have:
$$\varphi({\Lambda}, P) 
=(\lambda_{0}P, \ldots , \lambda_{k} P).$$
\end{nota}

 Let $A$ be a general point in 
$\sigma_{s}(Seg(\mathbb{P}^k\times X))$. Then there exist 
$s$ distinct points $\Lambda_1, \dots, \Lambda_s \in 
\mathbb{P}^k$ and $s$ distinct points $P_1, \dots, P_s \in X$ 
such that $A \in \langle  \varphi( \Lambda_1,P_1), \dots,  
\varphi( \Lambda_s,P_s)\rangle$.

Choose  a set of homogeneous  coordinates $(a_0,\dots,a_N)$ for  $A$.
 By a suitable choice of the  homogeneous coordinates 
 $(\lambda_{i,0},\dots, \lambda_{i,k}   )$ of the  points $\Lambda_i$,
 we can write: 
\begin{multline}\nonumber
A=(a_0,\dots,a_N)=  \varphi( \Lambda_1,P_1)+ \dots+  
\varphi( \Lambda_s,P_s) =\\ \nonumber =(\lambda_{1,0}P_1, 
\ldots , \lambda_{1,k}P_1 )+ \cdots +
(\lambda_{s,0}P_s \ldots , \lambda_{s,k} P_s).
\end{multline}
In the previous notation, this is equivalent to
$$A=(\lambda_{1,0}P_1+\dots+\lambda_{s,0}P_s, \dots, 
\lambda_{1,k}P_1+\dots+\lambda_{s,k}P_s).
$$

For any general point $A$ as above, we set: 
\begin{equation}
\Phi(A) = 
\langle \lambda_{1,0}P_1+ \cdots +  \lambda_{s,0}P_s,  \dots , 
\lambda_{1,k}P_1+ \cdots +  \lambda_{s,k}P_s\rangle .
\end{equation}
Observe that, since $A$ is general, then the right side of 
the equality represents a linear space of dimension 
$w= \min\{s-1,k\}$.

We want to show that, in this way, we get indeed a rational map
$$
\Phi:\sigma_{s}(Seg(\mathbb{P}^k\times X))\dashrightarrow  
GS_X(w,s).$$
 
The map $\Phi$ is well defined, if
$$(\alpha a_0,\dots,  \alpha a_N), \quad  \alpha \in K-\{0\}
$$
is another set of homogeneous  coordinates of $A$, then 
 $$\alpha(a_0,\dots,a_N)= 
 \alpha(\lambda_{1,0}P_1+\dots+\lambda_{s,0}P_s, \dots, 
\lambda_{1,k}P_1+\dots+\lambda_{s,k}P_s),
$$
 and in this case, obviously,
$ \lambda_{1,i}P_1+\dots+\lambda_{s,i}P_s $  and  
$  \alpha(\lambda_{1,i}P_1+\dots+\lambda_{s,i}P_s)$,  
for $0 \leq i \leq k$, represent the same projective points.
 
 Moreover, if there exist points
 $\M_i=(\mu_{i,0}, \dots, \mu_{i,k}) \in \mathbb P^k$ and 
 $Q_i =(y_{i,0},\dots, y_{i,r}  ) \in X$  
such that 
$$
A  =(a_0,\dots,a_N)= \varphi( \M_1,Q_1)+ \dots+  \varphi( \M_s,Q_s),
$$
 we get the following equality of  $(r+1)(k+1)$-tuples:
\begin{multline}\nonumber
(\lambda_{1,0}P_1+\dots+\lambda_{s,0}P_s, \dots, 
\lambda_{1,k}P_1+\dots+\lambda_{s,k}P_s) =
\\ =\nonumber
(\mu_{1,0}Q_1+\dots+ \mu_{s,0}Q_s, \dots, 
\mu_{1,k}Q_1+\dots+ \mu_{s,k}Q_s).
\end{multline}
Hence  
\begin{equation}\label{duepersegre}
\lambda_{1,i}P_1+\dots+\lambda_{s,i}P_s = 
\mu_{1,i}Q_1+\dots+ \mu_{s,i}Q_s; \ \ \ i=0,\dots,k.
\end{equation}
It follows that $\Phi$ is consistent.
\medskip

Next, we give a characterization of points belonging
to the inverse image $\Phi^{-1}(\Pi)$ of a space $\Pi\in GS_X(k,s)$.

\begin{lem} \label{stessiP}
Let $w=\min\{k,s-1\}$, $s-1\leq r$
 and take a general point
$ \Pi \in GS_X(w,s)$. Assume $\Pi\subset \langle P_1,\dots,P_s
\rangle$ for $P_1, \ldots , P_s\in X$ distinct points. Let $B$ be a general element in $\Phi^{-1}(\Pi)$.
Hence there exist points $\N_1, \dots, \N_s\in  \PP^k$ 
such that   
$$B=  \varphi( \N_1,P_1)+ \dots+  \varphi( \N_s,P_s).$$
\end{lem}
\begin{proof}

Let us stress, before beginning the proof, the meaning of ``general'',
in our setting. With  ``general'' point $\Pi \in GS_X (w, s)$,
(or ``general'' points of the fiber $\Phi^{-1}(\Pi)$)
we mean that, among the points of the Grassmannian (resp. the fiber), 
the ones for which the statement does not hold 
consist in a set of zero measure. More specifically, they sit
in a proper algebraic subvariety of $GS_X (w, s)$ (resp. $\Phi^{-1}(\Pi)$).
We notice that the generality hypothesis on $\Pi$ and $B$
are crucial, for the argument. In fact, for example, if $\Pi$ is not
general in 
$GS_X(k,s)$, we cannot say anything on the number $s$ of points. Moreover, 
without the generality hypothesis, we cannot properly define the map
$\Phi$, itself.

By definition, we know that there are points
$Q_1,\dots,Q_s\in X$ and $\M_1,\dots,\M_s\in \mathbb{P}^k$, with
$\M_i=(\mu_{i,0},\dots,\mu_{i,k})$ for $i=1, \ldots , s$, such that
\begin{multline}\nonumber
B= \varphi( \M_1,Q_1)+ \dots+  \varphi( \M_s,Q_s) =\\ \nonumber
=(\mu_{1,0}Q_1, \ldots , \mu_{1,k}Q_1 )+ \cdots +
(\mu_{s,0}Q_s \ldots , \mu_{s,k} Q_s).
\end{multline}
Since 
\begin{multline}\nonumber
\Phi((\mu_{1,0}Q_1, \ldots , \mu_{1,k}Q_1 )+ \cdots +
(\mu_{s,0}Q_s \ldots , \mu_{s,k} Q_s) ) =\\ \nonumber
=\langle \mu_{1,0}Q_1+ \cdots +  \mu_{s,0}Q_s,  \dots , 
\mu_{1,k}Q_1+ \cdots +  \mu_{s,k}Q_s\rangle 
\end{multline}
and 
$$\Pi =
\langle \lambda_{1,0}P_1+ \cdots +  \lambda_{s,0}P_s,  \dots , 
\lambda_{1,k}P_1+ \cdots +  \lambda_{s,k}P_s\rangle 
$$
it follows that each point $\mu_{1,i}Q_1+ \cdots + 
\mu_{s,i}Q_s $, ($i=0, \ldots , k$),
 lies in the span of the points $\lambda_{1,j}P_1+ \cdots +  
 \lambda_{s,j}P_s$, ($j=0, \ldots , k$). 

By the definition of $w$ and by the generality of $\Pi$, we may assume 
that the  points $\lambda_{1,j}P_1+ \cdots +  
\lambda_{s,j}P_s$, ($j=0, \ldots , w$)  are independent.
It follows that, both for $w=s-1$ and for $w=k$, 
there are coefficients $\alpha_{i,j} \in K$,  such that
\begin{multline}\nonumber
 \mu_{1,i}Q_1+ \cdots + \mu_{s,i}Q_s = \sum_{j=0}^w\alpha_{i,j}
 (\lambda_{1,j}P_1+ \cdots +  \lambda_{s,j}P_s) =\\ \nonumber
=\left(\sum_{j=0}^w\alpha_{i,j}\lambda_{1,j}\right)P_1+ \cdots + \left(\sum_{j=0}^w\alpha_{i,j}\lambda_{s,j}\right)P_s 
\end{multline}
for  $i=0, \ldots , k$.  

So, by setting $\nu_{h,i} = \left(\sum_{j=0}^w\alpha_{i,j}
\lambda_{h,j}\right)$, we  have
$$
 \mu_{1,i}Q_1+ \cdots + \mu_{s,i}Q_s = \nu_{1,i}P_1+ 
 \cdots + \nu_{s,i}P_s \ , \ \ \ \  i=0,\dots,k.$$
 
Hence we get: 
\begin{gather}\nonumber
B=(\mu_{1,0}Q_1+\dots+ \mu_{s,0}Q_s, \dots, 
\mu_{1,k}Q_1+\dots+ \mu_{s,k}Q_s) =\\ \nonumber
=  ( \nu_{1,0}P_1+ \cdots +  \nu_{s,0}P_s, \dots ,  
\nu_{1,k}P_1+ \cdots +  \nu_{s,k}P_s )= \\ \nonumber
=( \nu_{1,0}P_1, \dots, \nu_{1,k}P_1) + \dots +( 
\nu_{s,0}P_s, \dots, \nu_{s,k}P_s) =\\ \nonumber
=  \varphi( \N_1,P_1)+ \dots+  \varphi( \N_s,P_s),
\end{gather}
where  $\N_i =(\nu_{i,0},\dots,\nu_{i,k}) \in \PP^k$
for all $i$.
\end{proof}

\section{Some consequences on the identifiability of general points}\label{ident}

The previous construction of the map $\Phi$ in Section \ref{TheMap}, as well as
Lemma \ref{stessiP}, lead to the following analogue of the main
theorem in \cite{DF}, for identifiability.

 \begin{thm} \label{main} Let $k \leq s-1 <r$. The variety $X$ is  $(k,s)$\emph{-identifiable} if and only if $Seg(\mathbb{P}^k\times X)$ is $s$-\emph{identifiable}.
\end{thm}

\begin{proof}
Let  $ \Pi $ be a general element of $  GS_X(k,s)$. If there exist two
different sets of distinct points $\{P_1,\dots,P_s\},\{Q_1,\dots,Q_s\}\subset X$ such that
  $$ \Pi  \subset \langle P_1,\dots,P_s   \rangle \ \ \
\hbox{ and  } \ \ \ \Pi\subset \langle Q_1,\dots,Q_s\rangle ,$$
then, by Lemma \ref{stessiP}, for a general point $B$ in 
$\Phi^{-1}(\Pi)$, we have points $\M_i$'s and $\N_i$'s in $\mathbb{P}^k$, $i=1, \ldots , s$, with: 
$$B= \varphi( \M_1,Q_1)+ \dots+  \varphi( \M_s,Q_s)
= \varphi( \mathcal N_1,P_1)+ \dots+  \varphi( \mathcal N_s,P_s).
 $$
Since $\{P_1,\dots,P_s\} \neq \{Q_1,\dots,Q_s\} $, we get that 
$B$ lies in the span of two distinct sets of points 
of $Seg(\mathbb{P}^k\times X)$.

Now let $A$ be a  general element of $ \sigma_s( 
Seg(\mathbb{P}^k\times X)  )$. If
$$   A=\varphi( \Lambda_1,P_1)+ \dots+  \varphi( \Lambda_s,P_s)
=  \varphi( \M_1,Q_1)+ \dots+  \varphi( \M_s,Q_s)
$$
are two different decompositions of $A$,
then, by the definition of $\Phi$, 
$\Pi=\Phi(A)$  lies in the  span of the two sets of 
points  $\{P_1,\dots,P_s\}$ and $\{Q_1,\dots,Q_s\}$.
It suffices to prove that these two sets of points are distinct.

Since $A$ is general, we may assume that the two sets 
of points are both independent.
Since:
$$
\lambda_{1,i}P_1+\dots+\lambda_{s,i}P_s = 
\mu_{1,i}Q_1+\dots+ \mu_{s,i}Q_s, \quad i=0,\dots,k,
$$
then $\{P_1,\dots,P_s\} = \{Q_1,\dots,Q_s\}$
implies $\lambda_{j,i}=\mu_{j,i}$ for all $j,i$ (up to re-ordering the $P_i$'s and $Q_i$'s).
This contradicts the fact that the two decompositions of $A$
are different.
\end{proof}

\begin{cor}\label{sdentros} If the codimension of $X$ is  bigger 
than $s$, then $Seg(\mathbb{P}^{s-1}\times X)$ is $s$-identifiabile.
\end{cor}
\begin{proof} Enough to observe that the general $s$-secant 
$(s-1)$-space cuts $X$ only in $s$ points,
thus it is obvious that a general $(s-1)$-space contained in a  
$s$-secant $(s-1)$-space, is contained in just one of them!

Then, since  under our numerical assumptions we have $r-n >s$ 
(hence $s-1 <r$), 
we may use the previous theorem to get the conclusion.
\end{proof}

Using  Theorem 1.1 of \cite{BC}, which, in turn, is based on 
the main result of \cite{CGG} (namely Theorem 4.1), we are able to prove a criterion 
for the Grassmann identifiability.

\begin{thm}\label{tre}
Let $s$ be an integer such that 
$r>sn+s-1$ and $X$ is not $s$-defective.
Then for all integers $k>0$, $k\leq s-1$ such that  
$$sn + (k + 1)(s-1 - k)< (k + 1)(r - k)$$
the $(k,s)$-identifiability holds for $X$.
\end{thm}

\begin{proof} Theorem \ref{main} says that  
$(k,s)$-identifiability holds for $X$
when $s$-identifiability holds for $Seg(\mathbb{P}^k\times X)$.
Under our numerical assumptions, the $s$-secant variety of
$Seg(\mathbb{P}^k\times X)$ cannot cover the linear span of 
$Seg(\mathbb{P}^k\times X)$. Thus we may apply Theorem 1.1
of \cite{BC}, and conclude that 
$Seg(\mathbb{P}^k\times X)$ is $s$-identifiable.
\end{proof}

Indeed, the proof of Theorem \ref{main} can be
enhanced, to give the following, more precise
result:

\begin{prop} Let $w,\Pi,B$ be as in Lemma \ref{stessiP}.
The following two sets:
$$ \E(\Pi)=\{(P_1,\dots,P_s)\in X^s\; |\; \Pi\subset\langle P_1,\dots,P_s
\rangle\}$$
\begin{multline}\nonumber
 \E(B)=\{(P_1,\dots,P_s)\in X^s\; | \; \exists\; \N_1,\dots
,\N_s \in \mathbb{P}^k \hbox{ with } B\in\langle (\N_1,P_1),\dots,
(\N_s,P_s)\rangle\}
\end{multline}
have the same cardinality.
\end{prop}
\begin{proof} Almost immediate, following the proof of Theorem
\ref{main}. The unique
warning is that the set $\{P_1,\dots,P_s\}$ that we
use in the argument, must be 
independent. Since $\Pi,B$ are general,
they turn out to be independent
for {\it all} the elements
of $\E(B)$ or $\E(\Pi)$, when these sets are finite,
and for infinitely many elements, when they are infinite. 
\end{proof}

To be even more precise, the sets $\E(\Pi)$ and $\E(B)$ 
can be endowed with a quasi-projective structure and Lemma 
\ref{stessiP} shows indeed that
there exists a birational map $\E(\Pi)\to\E(B)$.

We will not explore this point of view any further, 
because we do not need it in the sequel.

 \section{Linear systems of tensors}\label{linsyst}
 
 In this section, we collect some consequences of the  previous 
 theory, trying to explain properly its range of application. 
\smallskip

We consider a vector space $V$ over $K$ of tensors of type $n_1+1, \ldots , n_t+1$, namely $V=V_1\otimes \cdots \otimes V_t$ 
where $V_i$  is a vector space of dimension $n_i+1$, for $i=1, \ldots , t$. A {\it linear
system} of tensors is just a linear subspace of $V$.
In the projective setting, tensors of type 
$n_1+1,\dots,n_t+1$, up to scalar multiplication, determine 
a projective space $\PP^M$, where $M= (\Pi_{i=1}^t (n_i+1))-1$.
A linear system of tensors is a linear subspace $\E$ of $\PP^M$.

We take the {\it dimension} of $\E$ to be the 
{\it projective} dimension of the linear subspace associated
to $\E$ (i.e. the affine dimension, minus $1$).
 
Inside the space of tensors, there is the subvariety $X$
of {\it decomposable tensors}, which corresponds
to the Segre embedding $X=Seg(\PP^{n_1}\times\dots\times\PP^{n_t})
\subset \PP^M$.
 
\begin{defn}
We say that the linear system $\E$ of tensors is
{\it computed} by $s$ decomposable tensors 
$P_1,\dots ,P_s\in X$ if for all $P \in \E$ there are scalars $a_1,\dots, a_s$
such that:
$$ P = a_1P_1+\cdots +a_sP_s.$$
Geometrically, this means that the linear space associated to
$\E$ lies in the span of the points $P_1,\dots, P_s$.

We say that $\E$ has {\it rank $s$} if $s$ is the minimum
such that there are $s$ tensors in $X$ which compute $\E$.

We say that a linear system $\E$ of rank $s$ is {\it identifiable}
if there exists a unique set of $s$ decomposable tensors, 
that compute $\E$.

We say that tensors of type $n_1+1,\dots,n_t+1$ are
{\it $(k,s)$-identifiable} if the general linear system $\E$
of such tensors, of dimension $k$ and rank $s$, is
identifiable. 
\end{defn}
 
It is immediate to see that the previous terminology
is consistent with the general terminology
of the paper, once one considers the linear
subspace associated to a linear system
(see also Definition 1.1 of \cite{Buczynski:2009fk}).
\smallskip

The map $\Phi$ constructed in the previous sections
maps a tensor $P$ of type $k+1,n_1+1,\dots,n_t+1$ to 
a linear system of dimension $k$ of tensors of
type $n_1+1,\dots,n_t+1$.
Roughly speaking, the map takes the tensor
$T$ to the linear space generated by its $k+1$ slices
along the first direction.

Thus, all the results in the previous section apply
to the identifiability of linear systems of tensors.
In particular, for instance, we see that:

\begin{rem}\label{sys} \
\begin{itemize}
\item[(i)]
The general linear systems of dimension $k$ of tensors
of type $n_1+1,\dots,n_t+1$ has rank $s$
if and only if $s$ is the minimum such that the secant variety
$\sigma_s(\PP^{k}\times\PP^{n_1}\times\dots\times\PP^{n_t})$
covers the projective space $\PP^N$, $N=(M+1)(k+1)-1$.
\item[(ii)]
There are exactly $q$ sets of decomposable tensors
that compute a general linear system of tensors of type
$n_1+1,\dots,n_t+1$ if and only if there are exactly $q$
decomposable tensors that compute a general
tensors of type $k+1,n_1+1,\dots,n_t+1$.  
\item[(iii)]
Tensors of type of type $n_1+1,\dots,n_t+1$
are $(k,s)$-identifiable if and only if
tensors of type $k+1,n_1+1,\dots,n_t+1$ are
$s$-identifiable.  
\end{itemize}
\end{rem} 
 
Let us see how the previous remarks allows to translate some known
facts about tensors to facts about linear systems of tensors. The next two examples are actually consequences of the gluing of the main results of \cite{CGG} and \cite{BC}.

\begin{ex} {\it For $m>4$, the general linear pencil of tensors of
type $2\times \dots\times 2$, ($m$-times) has rank 
$\lceil 2^m/(m+1)\rceil$.

The general linear pencil as above, of rank $s\leq 2^{m-1}/m$,
is identifiable.}

The first fact follows from the main result in \cite{CGG} (namely Theorem 4.1) that computes the dimension of the secant varieties of the Segre embedding of $\mathbb{P}^1\times \cdots \times \mathbb{P}^1$. The value $\lceil 2^m/(m+1)\rceil$ corresponds to the order of the secant variety that fills the ambient space. This result together with Remark \ref{sys} (i) proves the first fact. 

The second �fact follows from Remark \ref{sys} (ii) and the main result of \cite{BC} (Theorem 1.1) which shows that the Segre product of $(m+1)>5$ copies of $\mathbb{P}^1$'s is $k$-identifiable as soon as $\lceil 2^m/(m+1)\rceil$.
\end{ex}

\begin{ex} {\it The general linear pencil of tensors of
type $2\times 2\times 2\times 2$ has rank $6$.

The general linear pencil of tensors of
type $2\times 2\times 2\times 2$, of rank $s<5$,
is identifiable.

The general linear pencil of tensors of
type $2\times 2\times 2\times 2$, of rank $5$,
is {\rm NOT} identifiable: it is computed by exactly
two sets of decomposable tensors.}

Just use the main results in \cite{CGG} (Theorem 4.1) that gives the order of the secant variety of the Segre variety that covers the ambient space,
and Proposition 4.1 of \cite{BC} that explicitly says that the product of 5 copies of $\mathbb{P}^1$ is not 4-identifiable and moreover that through a general
point of the fifth secant variery one finds exactly two 5-secant, 4-spaces. These two results together with our Remark \ref{sys} leads to the example.
\end{ex}

There are also results for linear systems
of matrices, which, as far as we know,
cannot be found in the classical literature hence we quote the next two examples as surprising  new facts on matrices.

\begin{ex} {\it 
The general linear system of rank $s$ and dimension $c-1$,
of matrices of type $a\times b$, with $a\leq b\leq c$,
is identifiable, as soon as $s\leq ab/16$.}

It follows from the main result in \cite{CO} (Theorem 1.1) applied to our Remark \ref{sys} (iii). In fact  \cite[Theorem 1.1]{CO} states that the general tensor of $V_1\otimes V_2 \otimes V_3$ of rank $k$ has a unique decomposition if $k\leq 2^{\alpha +\beta -2}$ where $\alpha,\beta$ are the maximal integers such that $2^{\alpha}\leq \dim V_1$ and $2^{\beta}\leq \dim V_2$ and $\dim V_1 \leq \dim V_2 \leq \dim V_3$.
\end{ex}

\begin{ex}\label{4x4} {\it The general linear system of dimension $3$
of matrices of type $4\times 4$ has rank $7$.

The general linear system of dimension $3$
of matrices of type $4\times 4$ and rank $s<6$
is identifiable.

The general linear system of dimension $3$
of matrices of type $4\times 4$ and rank $s=6$
is {\rm NOT} identifiable: it is computed by exactly
two sets of decomposable tensors.}

Use the main results in \cite{AOP}, and   \cite[Theorem 1.3]{CO}. In particular \cite[Example 3.18]{AOP} shows that the Segre embedding of $\mathbb{P}^3\times \mathbb{P}^3\times \mathbb{P}^3$ is never defective, then its $7$-th secant variety fills the ambient space. While \cite[Theorem 1.3]{CO} says that a general tensor in $\mathbb{C}^4\times \mathbb{C}^4\times \mathbb{C}^4$ of rank 6 has exactly two decompositions. These two results glued together with our Remark \ref{sys} give the example.
\end{ex}

Tons of similar results, about the identifiability of linear
systems of tensors, can be found by rephrasing, from the
point of view of Remark \ref{sys} the examples
that the reader can find in \cite{Kr}, \cite{DeL}, \cite{Li},
\cite{CGG}, \cite{BC}, \cite{AOP}, \cite{CO}, \cite{BCO} etc.

We will not expound further on this subject.

\section{Some consequences on the dimension of secant
 varieties of Segre varieties}\label{dove}

The construction introduced with the map $\Phi$ in Section \ref{TheMap},
as well as the obvious remark at the beginning of the proof of Corollary
\ref{sdentros}, is indeed useful for the study of many aspects  of Segre
products. In this section, we would like to point out how the study of
the map can be used to determine the dimension of some secant variety.

Notice that the results of our Theorem \ref{dimsegre} correspond to Corollary 3.2 and Proposition 3.9 of 
 \cite{Buczynski:2009fk}.
The method of J. Buczynski and J.M. Landsberg is based on 
the invariant properties of a rational map $\pi$ (see \cite{Buczynski:2009fk} [Corollary3.6]), which is very close to
our $\Phi$.

We add this section because  we would like to re-organize the results,
showing how they follow from an elementary coordinate-based examination
of the map $\Phi$.

 \begin{thm} \label{mappaPhi}  Assume, as always,
$w=\min\{k,s-1\}$ and $s-1 \leq r$.
 Then we have:
$$\dim\sigma_s (Seg(\PP^k\times X))= 
\dim GS_X(w,s)+(w+1)(k+1)-1.
$$
 \end{thm}
 
\begin{proof}
Let $ \Pi $ be a general element of $GS_X(w,s)$, that is,
  $\Pi$ is a $w$-space contained in $\langle P_1,\dots,P_s\rangle ,$
where the $P_i$ are  independent  points of $X$.
  
If we prove that 
$\dim \Phi^{-1}(\Pi)= (w+1)(k+1)-1$, we are done. 

Even if $w<k$, we can fix scalars $ \lambda_{i,j} \in K$, 
with $i=1,\dots, s$ and $j=0,\dots, k$, such that
$$
\Pi =
\langle \lambda_{1,0}P_1+ \cdots +  \lambda_{s,0}P_s,  \dots , 
\lambda_{1,k}P_1+ \cdots +  \lambda_{s,k}P_s\rangle  .
$$
Consider the points $\Lambda_i= (\lambda_{i,0},\ldots, 
\lambda_{i,k})\in\mathbb P^k$, and let
\begin{equation} \label{spuntisegre}
A= \varphi( \Lambda_1,P_1)+ \dots+  \varphi( \Lambda_s,P_s) \in 
\sigma_{s}(Seg(\mathbb{P}^k\times X)).
\end{equation}
Obviously $A  \in \Phi^{-1} (\Pi)$ and, for a general choice
of the scalars, $A$ will be a general point  of $\Phi^{-1} (\Pi)$.
 
 Since $s\leq r+1$,   without loss of generality, 
we may assume that the $P_i$ are coordinate  points, say
\[
 P_1 = (1,0, \ldots ,0), P_2 = (0,1,\ldots ,0), 
 \dots ,P_s = (0, \ldots,0,1,\dots,0).
\]
With this choice of coordinates, it is easy to see that
$$\Phi (A)
=\langle
(\lambda_{1,0},\lambda_{2,0},  \dots,  \lambda_{s,0},0,  
\dots,0), \dots , (\lambda_{1,k}, \lambda_{2,k} ,\dots,  
\lambda_{s,k},0,  \dots,0) \rangle ,
$$

Now, fix another general point $B\in\Phi^{-1} (\Pi)$.
By Lemma \ref{stessiP}, we know that there are points 
$\M_i= (\mu_{i,0},\ldots, 
\mu_{i,k})\in\PP^k$ with
$$B=\varphi( \M_1,P_1)+ \dots+  \varphi( \M_s,P_s) $$
and so:
$$   \Phi (B) =  \langle
(\mu_{1,0}, \mu_{2,0},  \dots,  \mu_{s,0},0,  \dots,0), 
\dots ,
(\mu_{1,k}, \mu_{2,k} ,\dots,  \mu_{s,k},0,  \dots,0) \rangle.
$$
 
Since  $\Phi (A) = \Phi (B)$, in the case $w=k \leq s-1$, 
it follows that  each point  
$(\mu_{1,i}, \mu_{2,i},  \dots,  \mu_{s,i},0,  \dots,0)$, 
($ i=0, \ldots , k) ,$   lies in the span of the $k+1$ points 
 $(\lambda_{1,j},\lambda_{2,j},  \dots,  \lambda_{s,j},0,  
 \dots,0)$,  ($ j=0, \ldots , k ).$

In case $w=s-1<k$,  each point 
$(\mu_{1,i}, \mu_{2,i},  \dots,  
\mu_{s,i},0,  \dots,0)$ , ($i=0, \ldots , k $),
  lies in the span of  $w+1$ independent points among 
  the $k+1$ points    $(\lambda_{1,j},\lambda_{2,j}, 
   \dots,  \lambda_{s,j},0,  \dots,0)$, 
($ j=0, \ldots , k )$, and we may assume that these $w+1$ 
independent points are  $(\lambda_{1,j},\lambda_{2,j},  
\dots,  \lambda_{s,j},0,  \dots,0)$, with $ j=0, \ldots , w .$

In other words, there exist
 $(w+1)(k+1)$ elements  $\alpha_{i,j} \in K $ s.t.
$$(\mu_{1,i}, \mu_{2,i},  \dots,  \mu_{s,i},0,  \dots,0)=
 \sum_{j=0}^w \alpha_{i,j} (\lambda_{1,j},\lambda_{2,j},  
 \dots,  \lambda_{s,j},0,  \dots,0),
$$
where $i =0, \dots,k$.

Equivalently, the following linear system 
\[
\left (
\begin{matrix}
M & 0 & 0 & \dots & 0&0 \\
0 & M & 0 &  \dots & 0 &0  \\
\dots & \dots & \dots &\dots & \dots  & \dots\\
0 & \dots &0 &\dots & 0&M  \\
\end{matrix}
\right)
\left (
\begin{matrix}
\alpha_{0,0} \\
\dots  \\
\alpha_{0,k} \\
\alpha_{1,0} \\
\dots  \\
\alpha_{1,k} \\
\dots \\
\alpha_{k,0} \\
\dots  \\
\alpha_{k,k} \\
\end{matrix}
\right)=
\left (
\begin{matrix}
\mu_{1,0} \\
\dots  \\
\mu_{s,0} \\
\mu_{1,1} \\
\dots  \\
\mu_{s,1} \\
\dots \\
\mu_{1,k} \\
\dots  \\
\mu_{s,k} \\
\end{matrix}
\right)
\]
where
$M=
\left (
\begin{matrix}
\lambda_{1,0} & \dots & \lambda_{1,k } \\
\lambda_{2,0} & \dots & \lambda_{2,k } \\
\dots & \dots & \dots  \\
\lambda_{s,0} & \dots & \lambda_{s,k } \\
\end{matrix}
\right) $, has solutions. Since $A$ is general, 
the rank of the coefficient matrix of this linear system  
is $(w+1)(k+1)$.

Now, since 
$$B=  \varphi( \M_1,P_1)+ \dots+  \varphi( \M_s,P_s)$$
$$
=(\mu_{1,0}P_1, \ldots , \mu_{1,k}P_1 )+ \cdots +
(\mu_{s,0}P_s \ldots , \mu_{s,k} P_s) 
$$
$$=(\mu_{1,0},\mu_{2,0}\dots,  \mu_{s,0},0,\dots,0, 
\mu_{1,1},\mu_{2,1}\dots,  \mu_{s,1},0,\dots,0,$$
$$
\dots\dots,
\mu_{1,k},\mu_{2,k},\dots,  \mu_{s,k},0,\dots,0
),
$$
it immediately  follows  that the  dimension of 
$\Phi^{-1}(\Pi)$ is $(w+1)(k+1)-1$.
\end{proof}

{The previous argument shows that the map $\Phi$ has positive dimensional
fibers, in general.
It is interesting to observe that, nevertheless,  the structure of the
fibers yields that the identifiability
of $\Phi(A)$ implies the identifiability of $A$.
}

As  an easy consequence of Theorem \ref{mappaPhi}, 
we get  the following Terracini-type theorem (proved in \cite{DF}):

\begin{cor} \label{terracini} 
Let $k \leq s-1< r$. Then $X$ is $(k, s)$-defective with 
defect $\delta_{k,s}(X)=\delta$ if and only if 
$Seg(\mathbb{P}^k \times X)$ is $s$-defective with defect 
$\delta_{s}(Seg(\mathbb{P}^k \times X))=\delta$.
\end{cor}

\begin{proof}

{By Theorem \ref{mappaPhi}, and 
  a direct computation we  have
$$
\dim (\sigma_s (Seg(\mathbb{P}^k\times X)))-
\dim( GS_X(k,s))  =k^2 +2k
$$
$$=
exp\dim (\sigma_s (Seg(\mathbb{P}^k\times X)))-
exp\dim( GS_X(k,s))
$$
and we are done.}
\end{proof}
\medskip

Next, we get some  results about the defectivity 
or non-defectivity of  the $s$-th  higher secant variety 
of  $Seg(\PP^k\times X)$.

\begin{lem} \label {lemma1} 

For $s-1 < k $, and $s-1 \leq r$, 
we have
$$\dim \sigma_s(Seg(\PP^k \times X)) 
=\min \{ s(k+n+1)-1 ; s(k+r-s+2)-1 \} $$
\end{lem}

\begin {proof}
By Theorem \ref{mappaPhi} we get 
$$\dim  \sigma_s  (Seg(\PP^k\times X))= \dim GS_X(s-1,s)
 +s(k+1)-1.
$$

Since it is well known, (see, for instance, \cite[Section 2]{CC08}), 
that the dimension of $ GS_X(s-1,s)$ is the smallest 
between $sn$ and  the dimension of the Grassmannian 
$\mathbb{G}(s-1,r)$, the conclusion easily follows.
\end {proof}

\begin{thm} \label{dimsegre} Let $X\subset \mathbb{P}^r$ be an irreducible non-degenerate projective variety of dimension $n$.

\begin {itemize}
\medskip
\item [(i)] If $s-1 \geq r$,  then
 $$\sigma_s(Seg(\PP^k \times X)) =\PP^N,$$
so $\sigma_s(Seg(\PP^k \times X))$ is not defective.
\medskip

\item [(ii)]  Let $s -1<\min \{ r ; \ k\}$;
\begin{itemize}
\medskip

\item [(a)] if   $s-1 \leq r-n $, then
$$\dim \sigma_s(Seg(\PP^k \times X)) =s(k+n+1)-1,$$
and $\sigma_s(Seg(\PP^k \times X))$ is not defective;
\medskip

\item [(b)] if   $s -1 > r-n $,  then
$$\dim \sigma_s(Seg(\PP^k \times X)) =s(k+r-s+2)-1,$$
and $\sigma_s(Seg(\PP^k \times X))$ is  defective.
\end{itemize}
\medskip

\item[(iii)] If  $s-1=k < r , $  then
$$\dim \sigma_s(Seg(\PP^k \times X)) =
\min \{ s(k+n+1)-1 \ , \ N\},$$
and $\sigma_s(Seg(\PP^k \times X))$ is not defective.
\medskip

\item[(iv)] If $k < s-1 < r , $  then
$$\dim  \sigma_s(Seg(\PP^k\times X))= \dim GS_X(k,s)
 +k^2+2k.
$$
\end {itemize}
\end{thm}

\begin {proof}

(i) It is enough to prove this case for $s-1=r$. 
Let $P_1, \dots, P_s$ be independent points in $X$. 
We may assume that $
 P_1 = (1,0, \ldots ,0), P_2 = (0,1,\ldots ,0), 
 \dots ,P_s = (0, \ldots,0,1).
$
Hence for a general point $A=(\lambda_{1,0},\lambda_{2,0},  \dots,  \lambda_{s,0},  
\dots, \lambda_{1,k}, \lambda_{2,k} ,\dots,  \lambda_{s,k})  \in \PP^N$ we have 
$A= \varphi( \Lambda_1,P_1)+ \dots+  \varphi( \Lambda_s,P_s), $ where $\Lambda_i =(\lambda_{i,0}, \dots, \lambda_{i,k})
$.

(ii) By Lemma \ref{lemma1}  we immediately get the dimensions of 
$\sigma_s(Seg(\mathbb P^k \times X))$ both in case 
(a) and in case (b).

 Since in case (ii)(a) we get 
$N> s(k+n+1)-1$, 
 it follows that 
 $exp\dim \sigma_s(Seg(\PP^k \times X)) 
 =  s(k+n+1)-1,$
and so in this case $\sigma_s(Seg(\PP^k \times X))$ 
is not defective. 

In case (ii)(b) we have
$$s(k+n+1)-1-\dim \sigma_s(Seg(\PP^k \times X)) =s(n-r+s-1) >0,
$$
$$
N-\dim \sigma_s(Seg(\mathbb P^k \times X))
=(r-s+1)(k-s+1) >0.$$
Hence $\dim \sigma_s(Seg(\PP^k \times X))  < 
exp\dim \sigma_s(Seg(\PP^k \times X)) $.

(iii) For $s-1=k$, we have $ s(k+r-s+2)-1=N$, hence by Lemma 
\ref {lemma1} we get the conclusion.

(iv) Obvious from Theorem \ref{mappaPhi}.

\end {proof}

If $k  = r-n$, by applying the  theorem above we get the 
following interesting result.

\begin{cor}\label{55} \label{k=r-n}  If  $ k = r-n$, then 
$ \sigma_s (Seg(\mathbb{P}^k\times X))$ is never defective.
 \end{cor}

  \begin{proof}
{
  First assume $s-1=k$.
By Theorem   \ref{dimsegre} (iii) we get
  $$ \dim  \sigma_{s}  (Seg(\PP^{k}\times X))
=  \min \{ s(n+s)-1, \ s(r+1)-1 \} .
$$
Since $r-n= s-1$ we have
$$ s(n+s)-1 =  s(r+1)-1,$$ 
 $$s(n+s)-1 =   (k+1)(r+1)-1 = N ,
  $$
  $$s(r+1)-1= s(k+n+1)-1 ,$$
and  so
$$ \dim  \sigma_{s}  (Seg(\PP^{k}\times X))
=  N = s(k+n+1 )-1= s(\dim Seg(\PP^{k}\times X) +1 )-1.
$$
Now assume that $s\neq k+1$. In this case,
by
Remark \ref{TerrRem},
we get:
\begin{itemize}
\item
 for $s>k+1$, $ \dim  \sigma_{s}  (Seg(\PP^{k}\times X))
=  N$ ;
\item
 for $s<k+1$, $ \dim  \sigma_{s}  (Seg(\PP^{k}\times X))
=  s(k+n+1 )-1,$
\end{itemize}
and the conclusion  follows.}
\end{proof}

Next two examples, which are the only new spots in this section, show
how the previous analysis of the map $\Phi$, and the properties listed
above, allow us to settle some interesting facts about Segre varieties
and tensors.

\begin{ex}\label{ab}
Let $Y$ be the Segre-Veronese embedding of $ \PP^{{n+1} 
\choose 2}\times  \PP^n$ via divisors of bi-degree $(1,2)$. 
Then $\sigma_s Y$ is never defective. In fact,
let $X \subset \PP^{{n+2 \choose 2}-1}$ be the $2-$uple 
Veronese embedding of $\PP^n$. Since 
$$Y= Seg(\PP^{n+1 \choose 2}\times X).
$$
and since ${n+1 \choose 2}= {n+2 \choose 2}-1-n$, then from 
Corollary \ref{k=r-n} we get the conclusion.
\end{ex} 

We like to stress here that this proves one case of the Conjecture 5.2 of \cite{AB09a}  (see also the more general Conjecture 5.5 of \cite{AB09b}).

\begin{ex} \label{5.7} Let $Y$ be the Segre-Veronese embedding of $\mathbb{P}^k\times \mathbb{P}^{n_1}\times \cdots \times \mathbb{P}^{n_t}$ via divisors of multi-degree $(1,d_1, \ldots , d_t)$. If $k=\Pi_{i=1}^t{n_i+d_i \choose d_i}-\sum_{i=1}^t n_i -1$, Corollary \ref{k=r-n} implies that $\sigma_s(Y) $ is never defective.
\end{ex}

We end the paper with the following interesting remark pointed out by the anonymous referee that we thank for having shared it with us. In the previous example,  if we consider the particular case of $Y$ being the standard Segre embedding obtained by taking all $d_i$'s equal to 1, then we get kind of tensors of boundary format in the sense of Hyperdeterminants of Gelfand, Kapranov and Zelevinsky \cite {GKZ}[pp. 444--445].

\section*{Acknowledgments} 
All authors  were supported by MIUR funds. 
The second author was partially supported by  Project Galaad of INRIA So\-phia Antipolis M\'editerran\'ee 
(France)  and   Marie Curie Intra-European Fellowships for Career Development (FP7-PEOPLE-2009-IEF): ``DECONSTRUCT".

\bibliographystyle{alpha}


\begin{thebibliography}{\index}




\bibitem[AB09a]{AB09a}
Hirotachi Abo and Maria~Chiara Brambilla.
\newblock Secant varieties of Segre-Veronese varieties $\mathbb{P}^m \times \mathbb{P}^n$ embedded by the morphism given by $\mathcal{O}(1,2)$
\newblock {\em Experiment. Math.} 18(3):369--384, 2009.

\bibitem[AB09b]{AB09b}
Hirotachi Abo and Maria~Chiara Brambilla.
\newblock On the dimensions of secant varieties of segre-veronese varieties.
\newblock Preprint arXiv0912.4342v2, 12 2009.

\bibitem[AOP09]{AOP}
Hirotachi Abo, Giorgio Ottaviani, and Chris Peterson.
\newblock Induction for secant varieties of {S}egre varieties.
\newblock {\em Trans. Amer. Math. Soc.}, 361(2):767--792, 2009.

\bibitem[AMR09]{AMR}
Elizabeth~S. Allman, Catherine Matias, and John~A. Rhodes.
\newblock Identifiability of parameters in latent structure models with many
  observed variables.
\newblock {\em Ann. Statist.}, 37(6A):3099--3132, 2009.

\bibitem[BC11]{BC}
Cristiano Bocci and Luca Chiantini.
\newblock On the identifiability of binary segre products.
\newblock Preprint arXiv1105.3643, 05 2011.

\bibitem[BCO]{BCO}
Cristiano Bocci, Luca Chiantini, and Giorgio Ottaviani.
\newblock An inductive procedure for the identifiability of segre products.
\newblock In Preparation.

\bibitem[Br33]{Br} Jacob Bronowski. 
\newblock The sum of powers as simultaneous canonical espressions.
\newblock {\em Proc. Camb. Phyl. Soc.},  465--469, 1933.

\bibitem[BL11]{Buczynski:2009fk}
Jaroslaw Buczynski and Joseph M. Landsberg.
\newblock Ranks of tensors and a generalization of secant varieties.
\newblock Preprint arXiv 0909.4262v4, 11 2011.

\bibitem[CC03]{CaCh}
Enrico Carlini and Jaydeep Chipalkatti.
\newblock On {W}aring's problem for several algebraic forms.
\newblock {\em Comment. Math. Helv.}, 78(3):494--517, 2003.

\bibitem[CGG11]{CGG}
Maria~Virginia Catalisano, Anthony~V. Geramita, and Alessandro Gimigliano.
\newblock Secant varieties of {$\Bbb P^1\times\dots\times\Bbb P^1$}
  ({$n$}-times) are not defective for {$n\geq 5$}.
\newblock {\em J. Algebraic Geom.}, 20(2):295--327, 2011.

\bibitem[CC02]{ChCi}
Luca Chiantini and Ciro Ciliberto.
\newblock The {G}rassmannians of secant varieties of curves are not defective.
\newblock {\em Indag. Math. (N.S.)}, 13(1):23--28, 2002.

\bibitem[CC11]{MR2769871}
Luca Chiantini and Filip Cools.
\newblock Classification of {$(1,2)$}-{G}rassmann secant defective threefolds.
\newblock {\em Forum Math.}, 23(1):207--222, 2011.

\bibitem[CC01]{ChC}
Luca Chiantini and Marc Coppens.
\newblock Grassmannians of secant varieties.
\newblock {\em Forum Math.} 13:615-628, 2001.

\bibitem[CO11]{CO}
Luca Chiantini and Giorgio Ottaviani.
\newblock On generic identifiability of 3-tensors of small rank.
\newblock Preprint arXiv1103.2696v1, 03 2011.

\bibitem[CC08]{CC08}
Ciro Ciliberto and Filip Cools.
\newblock On {G}rassmann secant extremal varieties.
\newblock {\em Adv. Geom.}, 8(3):377--386, 2008.

\bibitem[Com02]{MR1931400}
Pierre Comon.
\newblock Tensor decompositions: state of the art and applications.
\newblock In {\em Mathematics in signal processing, {V} ({C}oventry, 2000)},
  volume~71 of {\em Inst. Math. Appl. Conf. Ser. New Ser.}, pages 1--24. Oxford
  Univ. Press, Oxford, 2002.

\bibitem[DF01]{DF}
Carla Dionisi and Claudio Fontanari.
\newblock Grassmann defectivity \`a la {T}erracini.
\newblock {\em Matematiche (Catania)}, 56(2):245--255 (2003), 2001.
\newblock PRAGMATIC, 2001 (Catania).

\bibitem[GKZ94]{GKZ}
Israel M. Gelfand, Mikhail Kapranov and Andrey V. Zelevinsky
\newblock {\em Discriminants, resultants, and multidimensional determinants}, 
\newblock  Birkhauser Boston Inc, 1994.

\bibitem[LC]{CL}
Lek-Heng Lim and Pierre Comon.
\newblock  Multiarray signal processing: Tensor decomposition meets compressed sensing
\newblock {\em C. R. Mecanique}, 338:311--320, 2010.

\bibitem[DL06]{DeL}
Lieven De~Lathauwer.
\newblock {A link between the canonical decomposition in multilinear algebra
  and simultaneous matrix diagonalization.}
\newblock {\em SIAM J. Matrix Anal. Appl.}, 28(3):642--666, 2006.

\bibitem[ERSS05]{ERSS}
Nicholas Eriksson, Kristian Ranestad, Bernd Sturmfels, and Seth Sullivant.
\newblock Phylogenetic algebraic geometry.
\newblock In {\em Projective varieties with unexpected properties}, pages
  237--255. Walter de Gruyter GmbH \& Co. KG, Berlin, 2005.

\bibitem[KB09]{KB}
Tamara~G. Kolda and Brett~W. Bader.
\newblock Tensor decompositions and applications.
\newblock {\em SIAM Rev.}, 51(3):455--500, 2009.

\bibitem[Kru77]{Kr}
Joseph~B. Kruskal.
\newblock Three-way arrays: rank and uniqueness of trilinear decompositions,
  with application to arithmetic complexity and statistics.
\newblock {\em Linear Algebra and Appl.}, 18(2):95--138, 1977.

\bibitem[Lic85]{Li}
Thomas Lickteig.
\newblock Typical tensorial rank.
\newblock {\em Linear Algebra Appl.}, 69:95--120, 1985.

\bibitem[Mel09]{Mella}
Massimiliano Mella.
\newblock Base loci of linear systems and the {W}aring problem.
\newblock {\em Proc. Amer. Math. Soc.}, 137(1):91--98, 2009.

\bibitem[Ter15]{Ter15}
Alessandro~Terracini.
\newblock Sulla rappresentazione delle coppie di forme ternarie mediante somme   di potenze di forme lineari.
\newblock {\em Ann. di Matem. pura ed appl.}, XXIV(III):91--100, 1915.




\end{thebibliography}

\end{document}